\documentclass[reqno]{amsart}
\usepackage{hyperref}

\def\N{{\mathbb{N}}}

\def\R{{\mathbb{R}}}

\def\B{{\mathbb{B}}}

\begin{document}
\title[rank theorem]{Rank theorem in infinite dimension and Lagrange multipliers}

\author[Blot]
{Jo\"el Blot}

\address{Jo\"el Blot: Laboratoire SAMM EA 4543,\newline
Universit\'{e} Paris 1 Panth\'{e}on-Sorbonne, centre P.M.F.,\newline
90 rue de Tolbiac, 75634 Paris cedex 13,
France.}
\email{blot@univ-paris1.fr}
\date{January, 28, 2018}
\begin{abstract} We use an extension to the infinite dimension of the Rank Theorem of the differential calculus to establish a Lagrange theorem for optimization problems in Banach spaces. We provide an application to variational problems on a space of bounded sequences under equality constraints. 
\end{abstract}
\maketitle
Classification MSC 2010: 49K27, 49K30.
\numberwithin{equation}{section}
\newtheorem{theorem}{Theorem}[section]
\newtheorem{lemma}[theorem]{Lemma}
\newtheorem{example}[theorem]{Example}
\newtheorem{remark}[theorem]{Remark}
\newtheorem{definition}[theorem]{Definition}
\newtheorem{proposition}[theorem]{Proposition}
\newtheorem{corollary}[theorem]{Corollary}
\section{Introduction}
Firstly we recall the extension of the rank theorem to the infinite dimension as it is established in \cite{Bl1} (Section 3).
\vskip1mm
\noindent
Secondly we establish a theorem of existence of Lagrange multipliers in Banach spaces for maximization problems under equality constraints (Section 4) like the following one
\[
(M) 
\left\{
\begin{array}{rl}
{\rm Maximize} & J(x)\\
{\rm subject\; to}& f(x) = 0
\end{array}
\right.
\]
where the functional $J$ and the mapping $f$ are defined on an open subset of a Banach space, and the mapping $f$ takes its values in a Banach space. Among the classical multiplier rules, when the differential of $f$ at the solution is surjective we obtain the existence of Lagrange multipliers without a multiplier before the criterion, i.e. $\hat{x}$ being a solution of $(M)$, we have $DJ(\hat{x}) = \lambda \circ Df(\hat{x})$ where $\lambda$ belongs to the topological dual space of the co-domain of $f$, and when we have only the closedness of the image of the differential of $f$ at the solution, there exists a multiplier rule with a multiplier before the criterion, i.e. $\lambda_0 DJ(\hat{x}) = \lambda \circ Df(\hat{x})$ with $(\lambda_0, \lambda) \neq (0,0)$ and $\lambda_0 \in \R$. Our aim is to establish a necessary optimality condition in the form  $DJ(\hat{x}) = \lambda \circ Df(\hat{x})$ wihtout assuming the surjectivity of $Df(\hat{x})$. To reach this aim we establish a result by using assumptions which are issued from a generalization to the infinite-dimensional Banach spaces of the classical Rank Theorem.
\vskip1mm
\noindent
Thirdly, in Section 5, we establish results on the space of the bounded sequences with values in a Banach space to prepare the following section.
\vskip1mm
\noindent
Lastly, in Section 6, we apply our theorem on Lagrange multipliers to a variational problem in infinite horizon and in discrete time on a space of bounded sequences under equality constraints.
%
\section{Notation}
When $E$ and $F$ are sets, when $f :E \rightarrow F$ is a mapping, when $E_0 \subset E$ and $F_0 \subset F$ are such that $f(E_0) \subset F_0$, we define the {\it abridgement} of $f$ (relatively to $E_0$ and $F_0$) as $ab f; E_0 \rightarrow F_0$ by setting $ab f(x) := f(x)$; this notation comes from \cite{RF} (p. 12).\\
When $X$ is a topological space, $\mathcal{O}(X)$ denotes the topology of $X$, and when $x\in X$, $\mathcal{O}_x(X)$ denotes the set of the open neighborhoods of $x$ in $X$.\\
The topological interior is denoted by $Int$.\\
When $E$ is a Banach space and $E_1$, $E_2$ are closed vector subspaces of $E$, the writting $E = E_1 \oplus^a E_2$ (respectively $E = E_1 \oplus E_2$) means the algebraic (respectively topological) direct sum, i.e. the mapping $(x_1, x_2) \mapsto x_1 + x_2$, from $E_1 \times E_2$ into $E$ is an isomorhism of vector spaces (respectively an isomorphism of topological vector spaces). When $E_1$ is a closed vector subspace of $E$, to say that $E_1$ is {\it topologically complemented} in $E$ means that there exists a closed vector subspace of $E$, say $E_2$, such that $E = E_1 \oplus E_2$.\\
The letter $D$ denotes the Fr\'echet differentiation, and when $i \in \{ 1,2 \}$, $D_i$ denotes the partial Fr\'echet differentiation of a mapping defined on a product space $E_1 \times E_2$ with respect to the $i^{\rm th}$ variable.
\vskip1mm
\noindent
When $X$ and $Y$ are real Banach spaces, $\mathcal{L}(X,Y)$ denotes the space of the linear continuous mappings from $X$ into $Y$, 
and if $L \in \mathcal{L}(X,Y)$, its norm is $\Vert L \Vert_{\mathcal{L}} := \sup \{ \Vert L(x) \Vert : x \in X, \Vert x \Vert \leq 1 \}$. The topological dual space of $X$, $\mathcal{L}(X, \R)$, is denoted by $X^*$. When $L \in \mathcal{L}(X,Y)$, the adjoint of $L$ is $L^* \in \mathcal{L}(Y^*, X^*)$ defined by $L^* \varphi := \varphi \circ L$ for all $\varphi \in Y^*$. When $M \subset X$, the orthogonal of $M$ is $M^{\perp} := \{ \varphi \in X^* : \forall x \in M, \varphi(x) = 0 \}$.
\vskip1mm
\noindent
When $A \in \mathcal{O}(X)$, $C^1(A,Y)$ denotes the space of the continuously Fr\'echet differentiable mappings from $A$ into $Y$.
When $f : A \rightarrow Y$ is a mapping and $y \in Y$, we set $[f = y] := \{ x \in A : f(x) = y \}$.\\
When $M \subset X$ and $a \in M$, the tangent space of $M$ at $a$ is $T_a M := \{ \alpha'(0) : \alpha \in C^1((- \epsilon, \epsilon), M), \alpha(0) = a \}$.
\vskip2mm
\noindent
We consider $\Omega$ as $\N$ or $\N_*  := \N \setminus \{ 0 \}$. When $E$ is a set, $E^{\Omega}$ denotes the space of the sequences defined on $\Omega$ with values in $E$.\\
An element of $E^{\Omega}$ will be denoted by $\underline{x} = (x_t)_{t \in \Omega}$.\\
When $E$ is a normed space, $\underline{x} \in {\ell}^{ \infty}(\Omega,E)$ denotes the set of the bounded sequences from $\Omega$ into $E$; we set $\Vert \underline{x}  \Vert_{\infty} := \sup_{ t \in \Omega} \Vert x_t \Vert$.\\
When $A \subset E$, ${\ell}_{\infty}(\Omega, A) := \{ \underline{x} \in {\ell}_{\infty}(\Omega, E) : \forall t \in \Omega, x_t \in A \}$.\\
We also write $B_{\infty}(\underline{x},r) := \{ \underline{u} \in {\ell}_{\infty}(\Omega,E) : \Vert \underline{u}  - \underline{x}  \Vert_{\infty} < r \}$.\\
${\ell}_1(\Omega, E)$ denotes the space of the sequences $\underline{x} = (x_t)_{t \in \Omega} \in E^{\Omega}$ such that\\
 $\sum_{t \in \Omega} \Vert p_t \Vert < + \infty$.\\
$c_0(\Omega, E)$ denotes the space of the sequences $\underline{x} = (x_t)_{t \in \Omega} \in E^{\Omega}$ such that \\
$\lim_{t \rightarrow + \infty} x_t = 0$.
§§§§§§§§§§§§§§§§§§§§§§§§§§§§§§§§§§§§§§§§§§§§§§§§
\section{The rank theorem in infinite dimension}
In this section we recall the Rank Theorem in infinite dimension and we establish several consequences of this theorem which are useful for the sequel. Under the assumptions of the Rank Theorem we describe the tangent space of a level set. As we indicate in \cite{Bl1}, there exist other generalizations of the classical Rank Theorem to the infinite-dimensional Banach spaces.
\vskip1mm
\noindent 
$X$ and $Y$ are real Banach spaces. The following result is given in \cite{Bl} and it is established in Theorem 1 and in Theorem 5 of \cite{Bl1}.
\begin{theorem}\label{th31} Let $A \in \mathcal{O}(X)$, $f \in C^1(A,Y)$, $\hat{x} \in A$, and we set $\hat{y} := f(\hat{x})$.\\
We assume that $E_2 := Ker Df(\hat{x})$ is topologically complemented in $X$ and $F_1 := Im Df(\hat{x})$ is closed and topologically complemented in $Y$; and so $X = E_1 \oplus E_2$ where $E_1$ is a closed vector subspace of $X$, and $Y = F_1 \oplus F_2$ where $F_2$ is a closed vector subspace of $Y$. We also assume that the following condition is fulfilled:
$$(*) \hskip5mm \exists A_0 \in \mathcal{O}_{\hat{x}}(A) \; {\rm s.t.} \; \forall x \in A_0, Im Df(x) \cap F_2 = \{0 \}.$$
Then the following assertions hold.
\begin{enumerate}
\item[(i)] $\exists V_1 \in \mathcal{O}_{D_1f_1(\hat{x})^{-1}(\hat{y}_1)}(E_1)$, $\exists V_2 \in \mathcal{O}_{\hat{x}_2}(E_2)$, $\exists B \in \mathcal{O}_{\hat{x}}(X)$, $\exists \psi : V_1 \times V_2 \rightarrow B$ a $C^1$ diffeormorphism,
$\exists W \in \mathcal{O}_{(\hat{y}_1,0)}(X)$, $\exists \Omega_1 \in \mathcal{O}_{\hat{y}_1}(F_1)$, $\exists \Omega_2 \in \mathcal{O}_{\hat{y}_2}(F_2)$, 
$\exists \phi : W \rightarrow \Omega_1 \times \Omega_2$ a $C^1$ diffeomorphism such that
$\phi^{-1} \circ f \circ \psi = Df(\hat{x})$ on $V_1 \times V_2$. 
\item[(ii)] $\exists G_1 \in \mathcal{O}_{\hat{x}_1}(E_1)$, $\exists G_2 \in \mathcal{O}_{\hat{x}_2}(E_2)$, $\exists \xi \in C^1(G_2,G_1)$ such that \\
$[f = \hat{y}] \cap (G_1 \times G_2) = \{ (\xi(x_2), x_2) : x_2 \in G_2 \}$.
\end{enumerate}
\end{theorem}
\vskip2mm
Note that the equality of the conclusion (i) can be rewritten as follows (cf. Theorem 1 of \cite{Bl1})
\begin{equation}\label{eq32}
\forall (x_1,x_2) \in V_1 \times V_2, \;\;  \phi^{-1} \circ f \circ \psi(x_1,x_2) = (D_1f_1(\hat{x}) x_1, 0)
\end{equation}
which a kind of local linearization of $f$.\\
When $X$ and $Y$ are finite-dimensional, the assumption $(*)$ of Theorem \ref{th31} is equivalent to the constancy of the rank of $Df(x)$ on a neighborhood of $\hat{x}$, cf. Proposition 4 in \cite{Bl1}. We can provide a heuristic meaning of the condition $(*)$: when $x$ is closed to $\hat{x}$, due to the continuity of $Df$, $Df(x)$ is closed to $Df(\hat{x})$ and $ImDf(x)$ cannot be "smaller" than $ImDf(\hat{x})$; this is a consequence of the openness of the set of the invertible linear continuous operators  in the normed space of the linear continuous operators. And so to have $ImDf(x)$ isomorphic to $ImDf(\hat{x})$, it suffices to forbid $ImDf(x)$ to be too "big". The condition $ImDf(x) \cap F_2 = \{ 0 \}$ is a way to forbid $ImDf(x)$ to be too "big". 
\vskip1mm
Now we describe consequences of Theorem \ref{th31}.
\begin{proposition}\label{prop32} In the setting and under the assumptions of Theorem \ref{th31}, the following assertions hold.
\begin{enumerate}
\item[(i)] $Df_2(\hat{x}) = 0$.
\item[(ii)] $D_2f_1 (\hat{x}) = 0$.
\item[(iii)] $\xi(\hat{x}_2) = \hat{x}_1$.
\item[(iv)] $D \xi(\hat{x}_2) = - D_1f_1(\hat{x})^{-1} \circ D_2f_1(\hat{x}) = 0$.
\end{enumerate}
\end{proposition}
\begin{proof} $\pi_2 : Y \rightarrow F_2$ and $p_2 : X \rightarrow E_2$ denote the projections.\\
 {\bf (i)} $f_2 = \pi_2 \circ f $ implies $Df_2(\hat{x}) = \pi_2 \circ Df(\hat{x}) = 0$ since $Im Df(\hat{x}) \cap F_2 = \{ 0 \}$.\\
{\bf (ii)} $D_2f_1(\hat{x}) = Df_1(\hat{x}) \circ p_2 = 0$ since $E_2 := Ker Df(\hat{x})$.\\
{\bf (iii)} From Theorem \ref{th31}(ii), since $f(\hat{x}_1, \hat{x}_2) = \hat{y}$, we necessarily obtain $\hat{x}_1= \xi(\hat{x}_2)$.\\
{\bf (iv)} Using (ii) of Theorem \ref{th31}, we have, for all $x_2 \in G_2$, $f_1(\xi (x_2), x_2) = \hat{y}_2$. Differentiating this equality with respect to $x_2$ at $\hat{x}_2$, we obtain 
$0= D_1f(\hat{x}) \circ D \xi (\hat{x}_2) + D_2f(\hat{x})$ which implies the announced formulas.
\end{proof}
\begin{proposition}\label{prop33} In the setting and under the assumptions of Theorem \ref{th31} we have $T_{\hat{x}} [f = \hat{y}] = Ker Df(\hat{x})$.
\end{proposition}
\begin{proof}
Let $v \in T_{\hat{x}} [f = \hat{y}]$. Hence there exists $c \in C^1((- \epsilon, \epsilon), X)$, where $\epsilon > 0$, such that $c((- \epsilon, \epsilon)) \subset  [f = \hat{y}]$, $c(0) = \hat{x}$ and $c'(0) = v$. Since $f(c(\theta)) = \hat{y}$ for all $\theta \in (-\epsilon, \epsilon)$, differentiating this equality with respect to $\theta$ at $0$, we obtain $0 = Df(c(0))c'(0) = Df(\hat{x}) v$, and so $v \in Ker Df(\hat{x})$. We have proven that $T_{\hat{x}} [f = \hat{y}] \subset Ker Df(\hat{x})$.
\vskip1mm
\noindent
To prove the inverse inclusion, we consider $v=(v_1,v_2) \in KerDf(\hat{x})$. Therefore we have $v_1 =0$ after the definition of $E_1$, $E_2$, and so $v=(0,v_2)$. Using Theorem \ref{th31}(ii), since $G_2$ is open, there exists $\epsilon > 0$ such that, for all $\theta \in (-\epsilon, \epsilon)$, we have $\hat{x}_2 + \theta v_2 \in G_2$. We define $c : (-\epsilon, \epsilon) \rightarrow X$ by setting $c(\theta) := (\xi(\hat{x}_2 + \theta v_2 ), \hat{x}_2 + \theta v_2 ) \in [f = \hat{y}] \cap (G_1 \times G_2)$. Note that we have $c(0) = (\xi(\hat{x}_2), \hat{x}_2) = \hat{x}$ after Proposition \ref{prop32}(iii). Note that we have $c'(0) = (D \xi(\hat{x}_2) v_2, v_2) = (0, v_2) = v$ after Proposition \ref{prop32}(iv) and so we have $v \in T_{\hat{x}} [f = \hat{y}]$. We have proven that $Ker Df(\hat{x}) \subset T_{\hat{x}} [f = \hat{y}]$, ansd so we have proven  the announced equality.
\end{proof}
\section{Lagrange multipliers in Banach spaces}
In this section we establish a theorem of existence of Lagrange multipliers for problem $(M)$ (written in Introduction) by using the Rank Theorem in infinite dimension. The interest of this result is to avoid a surjectivity on the differential of the equality constraint and nevertheless to avoid the presence of a multiplier before the criterion.
\vskip1mm
\noindent
Let $X$ and $Y$ be real Banach spaces, $A \in \mathcal{O}(A)$, $J : A \rightarrow \R$ be a functional, and $f : A \rightarrow Y$ be a mapping. We consider the problem $(M)$.
\begin{theorem}\label{th41}
Let $\hat{x}$ be a local solution of the problem $(M)$. We assume that the following conditions are fulfilled.
\begin{enumerate}
\item[(a)] $J$ is Fr\'echet differentiable at $\hat{x}$ and $f$ is of class $C^1$ on a neighborhood of $\hat{x}$.
\item[(b)] $E_2 := Ker Df(\hat{x})$ is topologically complemented in $X$; i.e. $X = E_1 \oplus E_2$ where $E_1$ is a closed vector subspace of $X$. $F_1 := Im Df(\hat{x})$ is closed and topologically complemented in $Y$; i.e. $Y = F_1 \oplus F_2$ where $F_2$ is a closed vector subspace of $Y$.
\item[(c)] There exists $A_0 \in \mathcal{O}_{\hat{x}}(A)$ such that, for all $x \in A_0$, $Im Df(x) \cap F_2 = \{0 \}$.
\end{enumerate}
Then there exists $\lambda \in Y^*$ such that $DJ(\hat{x}) = \lambda \circ Df(\hat{x})$.
\end{theorem}
\begin{proof}
Using Theorem \ref{th41}, we know that the set $[f = 0]$ is a $C^1$ manifold around $\hat{x}$, and since $J$ is of class $C^1$ at $\hat{x}$, the first-order necessary optimality condition is $DJ(\hat{x})h = 0$ when $h \in T_{\hat{x}} [f = 0]$. Using Proposition \ref{prop33}, we know that $T_{\hat{x}} [f = 0] = Ker Df(\hat{x})$, hence  we obtain $DJ(\hat{x})h = 0$ when $h \in KerDf(\hat{x})$, i.e., $DJ(\hat{x}) \in (KerDf(\hat{x}))^{\perp}$.
Since $ImDf(\hat{x})$ is closed, using Theorem 2.19 in \cite{Br} (p. 46), we have $(KerDf(\hat{x}))^{\perp} = ImDf(\hat{x})^*$, and then we have $DJ(\hat{x}) \in Im(Df(\hat{x}))^*$, hence there exists $\lambda \in Y^*$ such that $DJ(\hat{x}) = Df(\hat{x})^*\lambda = \lambda \circ Df(\hat{x})$.
\end{proof}
Notice that in Theorem \ref{th41} we have not a multiplier before the differential of the criterion. Such a result in finite-dimensional spaces is proven in \cite{Ja} (Proposition 2.3).
\section{On spaces of bounded sequences}
In this section we establish several results on the space of the bounded sequences to prepare the using of the theorem on the Lagrange multipliers of the previous section in such sequence spaces.
\vskip1mm
\noindent
First we recall a result on the characterization of the closedness of linear operators.
\begin{proposition}\label{prop51} Let $\mathcal{X}$ and $\mathcal{Y}$ be two real Banach spaces and $L \in \mathcal{L}(\mathcal{X}, \mathcal{Y})$. The two following assertions are equivalent.
\begin{itemize}
\item[(i)] $Im L$ is closed in $\mathcal{Y}$.
\item[(ii)] $\exists c(L) > 0, \forall y \in Im L, \exists x_y \in \mathcal{X}$ s.t. $L(x_y) = y$ and $\Vert x_y \Vert \leq c(L) \Vert y \Vert$.
\end{itemize}
\end{proposition}
We can find two different proofs of this result in \cite{BB} (Lemma 3.4) and in \cite{BC} (Lemma 2.1).
\begin{remark}\label{rem52} About the constant $c(L)$, under (i), when in addition we assume that $Ker L$ is topologically complemented in $\mathcal{X}$, i.e. $\mathcal{X} = Ker L \oplus \mathcal{X}_1$, where $ \mathcal{X}_1$ is a closed vector subspace of $\mathcal{X}$, we can consider the abridgement $ab L : \mathcal{X}_1 \rightarrow Im L$, $ab L(x) := L(x)$, which is an isomorphism of Banach spaces, and we have $c(L) = \Vert (ab L)^{-1} \Vert_{\mathcal{L}}$.
\end{remark}
Let $X$ and $Y$ be real Banach spaces. Let $(T_t)_{t \in \N_*} \in \mathcal{L}(X,Y)^{\N_*}$. We consider the three following conditions on this family of linear operators.
\begin{itemize}
\item[(C1)] $\sup_{t \in \N_*} \Vert T_t \Vert_{\mathcal{L}} < + \infty$.
\item[(C2)] For all $t \in \N_*$, $Ker T_t$ is topologically complemented in $X$, i.e. there exists a closed vector subspace of $X$, say $S_t$, such that $X = Ker T_t \oplus S_t$, and $Im T_t$ is closed and topologocally complemented in $Y$, i.e. there exists a closed vector subspace of $Y$, say $W_t$, such that $Y = Im T_t \oplus W_t$.
\item[(C3)] $\hat{c} := \sup_{t \in \N_*} c(T_t) < + \infty$.
\end{itemize}
Notice that (C2) is automatically fulfilled when $dim X < + \infty$ and $dim Y < + \infty$.\\
We associate to this family the four following sequence spaces.
$$\mathcal{K} := \{ \underline{x} \in {\ell}^{\infty}(\N_*, X) : \forall t \in \N_*, x_t \in Ker T_t \}.$$
$$\mathcal{S} := \{ \underline{x} \in {\ell}^{\infty}(\N_*, X) : \forall t \in \N_*, x_t \in S_t \}.$$
$$\mathcal{I} := \{ \underline{y} \in {\ell}^{\infty}(\N_*, Y) : \forall t \in \N_*, y_t \in Im T_t \}.$$
$$\mathcal{W} := \{ \underline{y} \in {\ell}^{\infty}(\N_*, Y) : \forall t \in \N_*, y_t \in W_t \}.$$
The condition (C1), which is equivalent to the boundeness of $(T_t(x_t))_{t \in \N_*}$ when $\underline{x}$ is bounded after the Banach-Steinhaus theorem, permits us to define the operator 
\begin{equation}\label{eqn51}
\mathcal{T} : {\ell}^{\infty}(\N_*,X) \rightarrow {\ell}^{\infty}(\N_*, Y), \hskip3mm \mathcal{T}(\underline{x}) := (T_y(x_t))_{t \in \N_*}.
\end{equation}
\begin{lemma}\label{lem53}
Under (C1, C2, C3) the following assertions hold.
\begin{itemize}
\item[(i)] ${\ell}^{\infty}(\N_*,X) = Ker \mathcal{T} \oplus \mathcal{S}$.
\item[(ii)] ${\ell}^{ \infty}(\N_*,Y) = Im \mathcal{T} \oplus \mathcal{W}$.
\end{itemize}
\end{lemma}
\begin{proof}
For each $t \in \N_*$, we consider the mappings $\pi_t : {\ell}^{\infty}(\N_*, X) \rightarrow X$ defined by $\pi_t(\underline{x}) := x_t$, and the mapping $\varpi_t : {\ell}^{\infty}(\N_*, Y) \rightarrow Y$, defined by $\varpi_t(\underline{y}) := y_t$. Clearly we have $\pi_t \in \mathcal{L}({\ell}^{\infty}(\N_*, X), X)$ and $\varpi_t \in \mathcal{L}({\ell}^{\infty}(\N_*, Y), Y)$. Note that we have
$$\mathcal{K} = \bigcap_{t \in \N_*} \pi_t^{-1}(Ker T_t), \mathcal{S} = \bigcap_{t \in \N_*} \pi_t^{-1}(S_t), \mathcal{I} = \bigcap_{t \in \N_*} \varpi_t^{-1}(Im T_t), \mathcal{W} = \bigcap_{t \in \N_*}\varpi_t^{-1}(W_t).$$
Due to the continuity of the $\pi_t$ and $\varpi_t$, these four spaces are intersections of closed subsets which implies that $\mathcal{K}$ and $\mathcal{S}$ are closed vector subspaces of ${\ell}^{\infty}(\N_*,X)$ and $\mathcal{I}$ and $\mathcal{W}$ are closed vector subspaces of ${\ell}^{\infty}(\N_*,Y)$.
\vskip1mm
It is easy to verify that
$${\ell}^{\infty}(\N_*X) = \mathcal{K} \oplus^a \mathcal{S}, \hskip3mm {\ell}^{\infty}(\N_*, Y) = \mathcal{I} \oplus^a \mathcal{W}.$$
Since ${\ell}^{\infty}(\N_*, X)$ and ${\ell}^{\infty}(\N_*,Y)$ are Banach spaces and since $\mathcal{K}$, $\mathcal{S}$, $\mathcal{I}$, $\mathcal{W}$ are closed vector subspaces, from the Inverse Mapping Theorem of Banach-Schauder, \cite{La} (Corollary 1.5, p. 388), we obtain
\begin{equation}\label{eq52}
{\ell}^{\infty}(\N_*, X) = \mathcal{K} \oplus \mathcal{S}, \hskip3mm {\ell}^{\infty}(\N_*,Y) = \mathcal{I} \oplus \mathcal{W}.
\end{equation}
We easily verify that $\mathcal{K} = Ker \mathcal{T}$, and from (\ref{eq52}) we obtain (i).
\vskip1mm
\noindent
About the image, we see that $Im \mathcal{T} \subset \mathcal{I}$. if $\underline{y} \in \mathcal{I}$ then, for all $t \in \N_*$, we have $y_t \in Im T_t$. From (C2), using Proposition \ref{prop51}, we can say that there exists $x_{t, y_t} \in X$ such that $T_t(x_{t,y_t}) = y_t$ and $\Vert x_{t,y_t} \Vert \leq c(T_t) \Vert y_t \Vert$. Hence, setting $\underline{x}_{\underline{y}} := (x_{t,y_t})_{t \in \N_*}$, we have $\mathcal{T}(\underline{x}_{\underline{y}}) = \underline{y}$ and $\Vert \underline{x}_{\underline{y}} \Vert_{\infty} \leq \hat{c} \Vert \underline{y} \Vert$, therefore $\underline{x}_{\underline{y}} \in {\ell}(\N_*, X)$ and $\mathcal{T}(\underline{x}_{\underline{y}}) = \underline{y}$, i.e. $\underline{y} \in Im \mathcal{T}$. We have proven that $\mathcal{I} \subset Im \mathcal{T}$, and consequently $\mathcal{I} = Im \mathcal{T}$
 From (\ref{eq52}) we obtain (ii).
\end{proof}
\begin{lemma}\label{lem54} Let $A \in \mathcal{O}(X)$. Then the following assertions hold.
\begin{enumerate}
\item[(i)]
$Int {\ell}^{\infty}(\N_*, A) = \{ \underline{x} \in {\ell}^{\infty}(\N_*, A) : \inf_{t \in \N_*} d(x_t, A^c) > 0 \}$, 
where\\
 $d(x_t, A^c) := \inf \{ \Vert x_t - z \Vert : z \in A^c \}$.
\item[(ii)] Let $\underline{x} \in Int {\ell}^{\infty}(\N_*, A)$. We set $r := \inf_{ t \in \N_*} d(x_t, A^c) > 0$. Then we have $\prod_{t \in \N_*} B(x_t,\frac{r}{2}) \subset Int{\ell}^{\infty}(\N_*,A)$.
\end{enumerate}
\end{lemma}
\begin{proof}
Assertion (i) is proven in \cite{BCr} (Lemma A.1.1). To prove (ii),\\
 if $\underline{u} \in \prod_{t \in \N} B(x_t,\frac{r}{2})$, then we have $\Vert u_t - x_t \Vert < \frac{r}{2}$, and for all $v \in A^c$, we have 
$$r < \Vert  x_t - v \Vert \leq \Vert x_t - u_t \Vert + \Vert u_t - v \Vert \leq  \frac{r}{2} + \Vert u_t - v \Vert \Longrightarrow r-  \frac{r}{2} \leq  \Vert u_t - v \Vert$$
which implies $d(u_t, A^c) \geq \frac{r}{2} > 0$, and so $\underline{u} \in Int{\ell}^{\infty}(\N_*, A)$.
\end{proof}
\begin{definition}\label{def55}
Let $A \in \mathcal{O}(X)$ and $g_t : A \rightarrow Y$ be a mapping for all $t \in \N_*$. The sequence $(g_t)_{t \in \N_*}$ is said uniformly equicontinuous on the bounded subsets of $A$ when
$$
\left\{
\begin{array}{l}
\forall B \in \B(A), \forall \epsilon > 0, \exists \eta_{B, \epsilon} > 0, \forall u, v \in B,\\
 \Vert u - v \Vert \leq \eta_{B, \epsilon} \Longrightarrow (\forall t \in \N_*, \Vert g_t(u) - g_t(v) \Vert \leq \epsilon).
\end{array}
\right.
$$
\end{definition}
\begin{proposition}\label{prop56} Let $A \in \mathcal{O}(X)$ and $g_t : A \rightarrow Y$ be a mapping for all $t \in \N_*$. We assume that the following condition is fulfilled: $\forall \underline{x} \in {\ell}^{\infty}(A), (g_t(x_t))_{t \in \N_*} \in {\ell}^{\infty}(\N_*, Y)$.\\
This condition permits to define the operator $G : {\ell}^{\infty}(\N_*, A) \rightarrow {\ell}^{\infty}(\N_*, Y)$ by setting $G(\underline{x}) := (g_t(x_t))_{t \in \N_*}$.\\
If, in addition we assume that $(g_t)_{t \in \N}$ is uniformly equicontinuous in the bounded subsets of $A$, then $G$ is continuous from $Int {\ell}^{\infty}(\N_*, A)$ into ${\ell}^{\infty}(\N_*, Y)$.
\end{proposition}
\begin{proof}
We arbitrarily fix $\underline{x} \in Int  {\ell}^{\infty}(\N_*, A)$. Hence there exists $r > 0$ such that $B_{\infty}(\underline{x}, r) \subset {\ell}^{\infty}(\N_*, A)$. \\
The set $B := \{ u \in A : \exists t_u \in \N_*, \Vert u - x_{t_u} \Vert < r \}$ is bounded since, when $u \in B$, $\Vert u \Vert \leq \Vert u - x_{t_u} \Vert + \Vert x_{t_u} \Vert \leq r + \Vert \underline{x} \Vert_{\infty}
< + \infty$. We arbitrarily fix $\epsilon > 0$. Let $\underline{z} \in Int  {\ell}^{\infty}(\N_*, A)$ such that $\Vert \underline{z} - \underline{x} \Vert_{\infty} < \min \{ r, \eta_{B, \epsilon} \}$ where $\eta_{B, \epsilon}$ is provided by Definition \ref{def55}. Hence we have $z_t \in B$ since $\Vert z_t - x_t \Vert < r$ for all $t \in \N_*$, and we have $\Vert z_t - x_t \Vert < \eta_{B, \epsilon}$ which implies $\Vert g_t(z_t) - g_t(x_t) \Vert  < \epsilon$. Hence we have $\Vert G(\underline{z}) - G(\underline{x}) \Vert_{\infty} \leq \epsilon$. We have proven that $G$ is continuous at $\underline{x}$.
\end{proof}
\begin{proposition}\label{prop57} Let $A \in \mathcal{O}(X)$ and $(g_t)_{t \in \N_*} \in C^1(A,Y)^{\N_*}$. We assume that the following conditions are fulfilled.
\begin{enumerate}
\item[(a)] $\forall \underline{x} \in Int {\ell}^{\infty}(\N_*, A)$, $(g_t(x_t))_{t \in \N_*} \in {\ell}^{\infty}(\N_*, Y)$.
\item[(b)] $\forall \underline{x} \in Int {\ell}^{\infty}(\N_*, A)$, $(Dg_t(x_t))_{t \in \N_*} \in {\ell}^{\infty}(\N_*, \mathcal{L}(X,Y))$.
\item[(c)] $(Dg_t)_{t \in \N_*}$ is uniformly equicontinuous on the bounded subsets of $A$.
\end{enumerate}
We consider the operator $G :  Int {\ell}^{\infty}(\N_*, A) \rightarrow {\ell}^{\infty}(\N_*, Y)$ defined by $G(\underline{x}) := (g_t(x_t))_{t \in \N_*}$. \\
Then the following assertions hold for all $\underline{x} \in Int  {\ell}^{\infty}(\N_*, A)$.
\begin{enumerate}
\item[(i)] $G$ is Fr\'echet differentiable at $\underline{x}$, and, for all $\underline{v} \in {\ell}^{\infty}(\N_*, X)$, we have\\
 $DG(\underline{x}) \underline{v} = (D g_t(x_t) v_t)_{t \in \N_*}$.
\item[(ii)] $G \in C^1(Int {\ell}^{\infty}(\N_*, A) , {\ell}^{\infty}(\N_*, Y))$.
\end{enumerate}
\end{proposition}
\begin{proof}  We arbitrarily fix $\underline{x} \in  Int {\ell}^{\infty}(\N_*, A)$. Using (b), the linear operator $\mathcal{T} : {\ell}^{\infty}(\N_*,X) \rightarrow {\ell}^{\infty}(\N_*Y)$, defined by $\mathcal{T} \underline{h} := (Dg_t(x_t) h_t)_{t \in _*\N}$, is well-defined and continuous.\\
We can use Lemma \ref{lem54}(ii) and assert that $\prod_{t \in \N_*} B(x_t, \frac{r}{2}) \subset  Int {\ell}^{\infty}(\N_*, A)$. Using the set $B := \{ u \in A : \exists t_u \in \N_*, \Vert u - x_t \Vert < r\}$, we have yet seen that $B$ is bounded in $A$ and we have $B(x_t, \frac{r}{2}) \subset B$.\\
We arbitrarily fix $\epsilon > 0$, and, using assumption (c),  we consider $\eta_{B, \epsilon} > 0$ provided by Definition \ref{def55} for the sequence $(Dg_t)_{t \in \N_*}$.\\
We arbitrarily fix $\underline{h} \in{\ell}^{\infty}(\N_*, X)$ such that $\Vert \underline{h} \Vert_{\infty} < \min \{ \frac{r}{2}, \eta_{B, \epsilon} \}$. Hence, $\forall t \in \N_*$, $\forall z_t \in [x_t, x_t + h_t] \subset B(x_t, \frac{r}{2})$, we have $\Vert z_t - x_t \Vert < \eta_{B, \epsilon}$ which implies $\Vert Dg_t(z_t) - Dg_t(x_t) \Vert_{\mathcal{L}} < \epsilon$.\\
Now using the Mean Value Inequality as established in \cite{ATF} (Corollary 1, p. 144), we obtain, for all $t \in \N_*$, \\
$\Vert g_t(x_t + h_t) - g_t(x_t) - Dg_t(x_t)h_t \Vert \leq  \sup_{z_t \in [x_t, x_t + h_t]} \Vert Dg_t(z_t) - Dg_t(x_t) \Vert_{\mathcal{L}} \Vert h_t \Vert \leq \epsilon \Vert h_t \Vert$, and taking the sup on the $t \in \N_*$, we obtain 
$$\Vert G(\underline{x} + \underline{h} - G(\underline{x}) - \mathcal{T} \underline{h} \Vert_{\infty} \leq \epsilon \Vert \underline{h} \Vert_{\infty}.$$
Hence we have proven that $G$ is Fr\'echet differentiable at $\underline{x}$ and that $DG(\underline{x}) \underline{h} = (Dg_t(x_t)h_t)_{t \in \N_*}$. \\
Applying Proposition \ref{prop56} to $DG$, we obtain the continuity of $DG$.
\end{proof}
\begin{proposition}\label{prop58} Let $A \in \mathcal{O}(X)$ and $(g_t)_{t \in \N_*} \in C^1(A,Y)^{\N_*}$. We assume that the conditions (a), (b), (c) of Proposition \ref{prop57} are fulfilled. 
Let $\hat{\underline{x}} = (\hat{x}_t)_{t \in \N_*} \in Int {\ell}^{\infty}(\N_*, A)$. We assume that the following conditions are fulfilled.
\begin{itemize}
\item[(d)] $\sup_{t \in \N_*} \Vert Dg(\hat{x}_t) \Vert_{\mathcal{L}} < + \infty$.
\item[(e)] For all $t \in \N_*$ there exist a closed vector subspace of $X$, say $S_t$, and a closed vector subspace of $Y$, say $W_t$, such that $X = Ker Dg(\hat{x}_t) \oplus S_t$ and $Y = Im Dg(\hat{x}_t) \oplus W_t$.
\item[(f)] $\hat{c} := \sup_{t \in \N_*} c(Dg(\hat{x}_t)) < + \infty$.
\end{itemize}
Then we have ${\ell}^{\infty}(\N_*,X ) = Ker DG(\hat{\underline{x}}) \oplus \mathcal{S}$ and ${\ell}^{\infty}(\N_*, Y) = Im DG(\hat{\underline{x}}) \oplus \mathcal{W}$, with $\mathcal{S} := \{ \underline{v} \in {\ell}^{\infty}(\N_*, X) : \forall t \in \N_*, v_t \in S_t \}$ and $\mathcal{W} := \{ \underline{w} \in {\ell}^{\infty}(\N_*,Y) : \forall t \in \N_*, w_t \in W_t \}$.
\end{proposition}
\begin{proof} This result is a consequence of Proposition \ref{prop57} and of Lemma \ref{lem53} with $T_t = Dg(\hat{x}_t)$.
\end{proof}
\noindent
Notice that (e) is automatically fulfilled when $dim X < + \infty$ and $dim Y < + \infty$.

\section{A variational problem}
In this section we consider a maximization  problem in infinite horizon and in discrete time under holonomic constraints. The unknown variable is a bounded sequence with values in a real Banach space $X$. Using the results of the previous sections we obtain a first-order necessary optimality condition in the form of a nonhomogeneous Euler-Lagrange equation.
\vskip1mm
\noindent
Let $A$ be a nonempty subset of $X$. For all $t \in \N$, let 
$u_t : A \times A \rightarrow \R$ and for all $t \in \N_*$ let $g_t : A \rightarrow Y$ be functions, where $Y$ is a real Banach space. We fix a vector $\sigma \in A$, a real number $\beta \in (0,1)$ and we consider the following variational problem under holonomic constraints. 
\[
(\mathcal{V})
\left\{
\begin{array}{rl}
{\rm Maximize} & J(\underline{x}) := \sum_{t=0}^{+ \infty}\beta^t u_t(x_t, x_{t+1})\\
{\rm subject \; to} & \underline{x} \in {\ell}^{\infty}(\N, A), x_0 = \sigma\\
{\rm and}  & \forall t \in \N_*, g_t(x_t) = 0.
\end{array}
\right.
\]
The problems of Calculus of Variations or of Optimal Control in discrete time and in infinite horizon are very usual in Economics and in Management; see for instance \cite{BH}, \cite{BB}, \cite{BCr} and references therein.
\begin{theorem}\label{th61}
Let $\underline{\hat{x}}$ be a solution of $(\mathcal{V})$. We assume that the following conditions are fulfilled
\begin{enumerate}
\item[(A1)] $(\hat{x}_t, \hat{x}_{t+1})_{t \in \N} \in Int {\ell}^{\infty}(\N,A \times A)$.
\item[(A2)] $\forall t \in \N$, $u_t \in C^1(A \times A, \R)$.
\item[(A3)] $\forall (\underline{x}, \underline{y}) \in Int{\ell}^{\infty}(\N, A \times A)$, 
$(u_t(x_t,y_t))_{t \in \N} \in {\ell}_{\infty}(\N, \R)$.
\item[(A4)] $\forall (\underline{x}, \underline{y}) \in Int{\ell}^{\infty}(\N,A \times A)$, $(Du_t(x_t,y_t))_{t \in \N} \in {\ell}^{\infty}(\N, (X \times X)^*)$.
\item[(A5)] $(Du_t)_{t \in \N}$ is uniformly equicontinuous on the bounded subsets of $A \times A$.
\item[(A6)] $\underline{\hat{x}} \in Int{\ell}^{\infty}(\N, A)$.
\item[(A7)] $\forall t \in \N_*$, $g_t \in C^1(A, \R)$.
\item[(A8)] $\forall \underline{x} \in Int{\ell}^{\infty}(\N_*,A)$, $(g_t(x_t))_{t \in \N_*} \in {\ell}^{\infty}(\N_*, Y)$.
\item[(A9)] $\forall \underline{x} \in Int{\ell}^{\infty}(\N_*,A)$, $(Dg_t(x_t))_{t \in \N_*} \in {\ell}^{\infty}(\N_*, \mathcal{L}(X, Y))$.
\item[(A10)] $(Dg_t)_{t \in \N_*}$ is uniformly equicontinuous on the bounded subsets of $A$.
\item[(A11)] For all $t \in \N_*$, there exists a closed vector subspace of $X$, say $S_t$, such that $X = Ker Dg_t(\hat{x}_t) \oplus S_t$, and there exists a closed vector subspace of $Y$, say $W_t$, such that $Y = Im Dg_t(\hat{x}_t) \oplus W_t$.
\item[(A12)]  $sup_{t \in \N_*} c(Dg_t(\hat{x}_t) ) < + \infty$.
\item[(A13)] $\exists r > 0$, $\forall \underline{x} \in B_{\infty}(\underline{\hat{x}}, r)$,   $\forall t \in \N_*$,
$ Im Dg_t({x}_t) \cap W_t = \{ 0 \}$.
\end{enumerate}
Then there exists a sequence $(p_t)_{t \in \N_*} \in {\ell}^1({\N_*}, Y^*)$ which satisfies, for all $t \in \N_*$, the following equality
$$\beta^t D_1 u_t(\hat{x}_t, \hat{x}_{t+1}) + \beta^{t-1} D_2 u_{t-1}(\hat{x}_{t-1}, \hat{x}_{t})= p_t \circ D g_t(\hat{x}_t).$$
\end{theorem}
\begin{proof} For each $t \in \N_*$ we consider the following static problem.
\[
(\mathcal{R}_t)
\left\{
\begin{array}{rl}
{\rm Maximize} & J_t(x_t) := \beta^{t-1} u_{t-1}(\hat{x}_{t-1}, x_t) + \beta^t u_t(x_t, \hat{x}_{t+1})\\
{\rm subject \; to} & x_t \in A, \hskip2mm g_t(x_t) = 0.
\end{array}
\right.
\]
Since $\underline{\hat{x}}$ is a solution of ($\mathcal{V}$), proceeding by contradiction, it is easy to see that $\hat{x}_t$ is a solution of ($\mathcal{R}_t$). Under our assumptions we can use Theorem \ref{th41} on
($\mathcal{R}_t$) and assert that there exists $q_t \in Y^*$ such the following equality holds.
\begin{equation}\label{eq61}
\beta^{t-1} D_2 u_{t-1}\hat{x}_{t-1}, \hat{x}_t) + \beta^t u_t(\hat{x}_t, \hat{x}_{t+1}) = q_t \circ Dg_t(\hat{x}_t).
\end{equation}
Under (A11), we know that, for every $y \in Y$, there exists a unique $(z,w) \in ImDg_t(\hat{x}_t) \times W_t$ such that $y = z + w$. We define $p_t \in Y^*$ by setting $p_t(y) := q_t(z)$. Hence we deduce from (\ref{eq61}) the following equality:
\begin{equation}\label{eq62}
\beta^{t-1} D_2 u_{t-1}(\hat{x}_{t-1}, \hat{x}_t) + \beta^t D_1u_t(\hat{x}_t, \hat{x}_{t+1}) = p_t \circ Dg_t(\hat{x}_t).
\end{equation}
Now we ought to prove that $(p_t)_{t \in \N_*} \in {\ell}_1(\N_*, Y^*)$.
\vskip1mm
\noindent
Let $\underline{y} \in {\ell}_{\infty}(\N_*, Y)$. Using Proposition \ref{prop58}, there exists $\underline{x} \in {\ell}_{\infty}(\N_*, X)$ and $\underline{w} \in \mathcal{W}$ such that $\underline{y} = DG(\underline{\hat{x}})\underline{x} + \underline{w}$, i.e. $y_t = Dg_t(\hat{x}_t) x_t + w_t$. Hence we have
\[
\begin{array}{ccl}
p_t(y_t) & = & p_t \circ Dg_t(\hat{x}_t) x_t + p_t(w_t)\\
\null & = & p_t \circ Dg_t(\hat{x}_t) x_t + 0\\
\null & = & \beta^{t-1} D_2 u_{t-1}(\hat{x}_{t-1}, \hat{x}_t)x_t + \beta^t D_1u_t(\hat{x}_t, \hat{x}_{t+1})x_t
\end{array}
\]
The sequences $(D_2 u_{t-1}(\hat{x}_{t-1}, \hat{x}_t))_{t \in \N_*}$ and $( D_1u_t(\hat{x}_t, \hat{x}_{t+1}))_{ t \in \N_*}$ are bounded after (A4), and since $(x_t)_{t \in \N_*}$ is bounded, the sequences 
$(D_2 u_{t-1}(\hat{x}_{t-1}, \hat{x}_t)x_t)_{t \in \N_*}$ and $( D_1u_t(\hat{x}_t, \hat{x}_{t+1})x_t)_{ t \in \N_*}$ are bounded.\\
 Since $\beta \in (0,1)$, the series $\sum_{ t\geq 1} \beta^{t-1} D_2 u_t(\hat{x}_{t-1}, \hat{x}_t)x_t$ and $\sum_{t \geq 0} \beta^t D_1u_t(\hat{x}_t, \hat{x}_{t+1})x_t$ are convergent in $\R$, and consequently the series $\sum_{t \geq 1} p_t(y_t)$ is convergent in 
$ \R$.  Since $c_0(\N_*, Y) \subset {\ell}_{\infty}(\N_*, Y)$, we obtain that the series $\sum_{t \geq 1} p_t(y_t)$ is convergentin $\R$  when $\underline{y} \in c_0(\N_*, Y)$. Using \cite{Le}(assertion ($\alpha$) in page 247), we can assert that $(p_t)_{ t \in \N_*} \in {\ell}_1(\N_*, Y^*)$.
\end{proof}
Notice that (A3), (A4), (A8), (A9) and (A11) are automatically fulfilled when $dim X < + \infty$ and $dim Y < + \infty$. When $dim Y < + \infty$, using a represention of the topological dual space of ${\ell}_{\infty}(\N_*, Y^*)$ (as a direct of ${\ell}_1(\N_*, Y^*)$ and of another subspace) which is given in \cite{AB} Chapter 15, Section 15.8) we can process as in \cite{BH} (Chapter 3) and to obtain the sequence $(p_t)_{tb \in \N_*}$ as the component in ${\ell}_1(\N_*, Y^*)$ of an element of ${\ell}_{\infty}(\N_*, Y^*)^*$.

\vskip2mm
\noindent
{\bf Acknowledgements.} I thank the reviewers very much for helping me improve the contents of the paper, especially the one who kept me from making a mathematical mistake and whose suggestions were very useful and constructive.

\end{document}